\documentclass[12pt]{article}
\usepackage{layout}
\usepackage{amsmath}

\newtheorem{thm}{Theorem}[section]
\newtheorem{prop}[thm]{Proposition}
\newtheorem{lem}[thm]{Lemma}
\newtheorem{dfn}[thm]{Definition}
\newtheorem{cor}[thm]{Corollary}

\newtheorem{rem}[thm]{Remark}

\newcommand{\qed}{\hfill \fbox{}\medskip}

\begin{document}

\title{The $h$-expansion of Macdonald operators and their expression by Dunkl operators}
\author{Hidekazu Watanabe}
\date{March 15, 2012}
\maketitle

~~~~~~~~~~~~~~~~~~~e-mail; watanabe(underbar)hi(underbar)de@yahoo.co.jp\\
~~~~~~~~~~~~~~~~~~~~~~~~~~~~~~~~~~~~Aizu Gakuho High School. \\
~~~~~~~~~~~~~~~~~~~~~~~~~~1-1 ikki-town ooaza-yawata-aza-yawata \\
~~~~~~~~~~~~~~~~~~~~~Aizuwakamatsu-city Fukushima 965-0003, Japan. \\
~~~~~~~~~~~~~~~~~~~~~~~~~~TEL: 0242-22-3491,  Fax: 0242-22-3521
\begin{abstract}
Macdonald symmetric functions have two parameters $q,t$ and they contain Jack symmetric functions and
Schur symmetric functions as the specialization of these paremeters. Macdonald symmetric functions can be 
defined as the eigenfunctions of Macdonald operators which consist of a commutative family \cite{M}. 
By the specialization of the parameters,
which means to take the limit $h$ to $0$ on condition
 that $q=exp(h),  t=exp(\beta h)$, we can see Jack symmetric functions appear. From the viewpoint of 
operators,  when we calculate the Taylor expansion of Macdonald operators around 
$h=0$ supposing that $q=exp(h),  t=exp(\beta h)$, we can see Dunkl operators appear as the coefficient of 
$h^1$ and $h^2$. These Dunkl operators also consist of commutative family which act on the space of symmetric 
functions, and have Jack symmetric functions as eigenfunctions \cite{DW}.
The degree-two-Dunkl operator is the same
as the 'Calogero-Sutherland-Hamiltonian' and it appears as the coefficient of $h^2$ of the h-expansion 
of Macdonald operators. Therefore, it is natural to expect that there appears higher degree Dunkl operator
as the coefficient of higher order of $h$. Although we can ascertain that the coefficients of $h^i, (i\geq 4)$
don't always commute with the coefficients of lower order, we can verify that the coefficients of within 
the third degree  always commute each other by writing them down as the polynomials of Dunkl operators. 
The object of this paper is to exhibit the calculation minutely, especially the third order of 
Macdonald operator's $h$-expansion.  
\end{abstract}
\section{Introduction}
Macdonald operators are defined as follows, which act on the space of symmetric functions 
$\bf{Q}(q,t)\otimes\Lambda_n$.
\begin{equation*}
D_n^r= t^{r(r-1)/2}\sum_{|I|=r}\prod_{i\in I,j\notin I}\frac{t x_i-x_j}{x_i-x_j}
\prod_{i\in I} T_{q,x_i} 
\end{equation*}
\begin{equation*}
T_{q,x_i} f(x_1,\cdots,x_i,\cdots,x_n)=f(x_1,\cdots,qx_i,\cdots,x_n)
\end{equation*}
$t$ and $q$ are indeterminant. $I$ runs all subsets of $\{1,\cdots,n\}$ which have $r$ elements.
The fundamental reference is $\cite{M}$. Next, we have to achieve the $h$-expansion of these operators
supposing that $q=exp(h),  t=exp(\beta h)$. So, the operators which appear as $h^k$'s 
coefficients have still one parameter $\beta$. First $\frac{t x_i-x_j}{x_i-x_j}$ is expanded as
\begin{equation*}
\frac{(1+\beta h+\frac{\beta^2 h^2}{2}+\cdots)x_i-x_j}{x_i-x_j}=
1+\frac{x_i}{x_i-x_j}\sum_{k=1}^{\infty}(\beta^k h^k/k!)
\end{equation*}
And $T_{q,x}$ acts on the $x^m$ as 
\begin{eqnarray*}
T_{q,x}x^m= q^m x^m= (\sum_{k=0}^{\infty}m^k h^k/k!)x^m
= (\sum_{k=0}^{\infty}(x \partial_{x})^k h^k/k!)x^m
\end{eqnarray*}
Therefore we can ascertain the following expansion
\begin{equation}
T_{q,x_i}=\sum_{k=0}^{\infty}(x_i \partial_{x_i})^k h^k/k!=\exp(h(x_i \partial_{x_i}))
\end{equation}
Furthermore, the operator $x_i\partial_{x_i}$ and $x_j\partial_{x_j}$ commute, we can 
also conclude the following.
\begin{equation}
T_{q,x_1}T_{q,x_2}\cdots T_{q,x_m}=\exp(h(x_1\partial_{x_1}+\cdots+x_m\partial_{x_m})).
\end{equation}
We often use the notation $\partial_{i}$ instead of $\partial_{x_i}$.

When we are engaged in the calculation of the coefficient of $h^m$, the part free from any
derivation operator can be calculated ignoring all $q$-shift operators $T_{q,x_i}$. This 
scalor part is 'the $\beta^m$-coefficient' of the $h^m$-coefficient in fact. Therefore it is 
very convenient to calculate the scalor part beforehand by ignoring all  $T_{q,x_i}$s. Then,
naturally there appears 't-binomial'. To clarify its h-expasion is the main object of next
section.
\section{The $h$-expansion of t-binomials}
\begin{dfn}
The t-binomial is defined as 
$$
\begin{bmatrix}n\\r\end{bmatrix}_t=\frac{1-t^n}{1-t}\frac{1-t^{n-1}}{1-t^2}
\cdots \frac{1-t^{n-r+1}}{1-t^r}.
$$
\end{dfn}
We often use the notation $\begin{bmatrix}n\\r\end{bmatrix}$ instead of  
$\begin{bmatrix}n\\r\end{bmatrix}_t$.
The following lemma is directly proved from the definition.
\begin{lem}
\begin{eqnarray}
\begin{bmatrix}n\\r\end{bmatrix}_t&=&\begin{bmatrix}n-1\\r-1\end{bmatrix}_t
+t^r\begin{bmatrix}n-1\\r\end{bmatrix}_t\label{bino}\\
&=&t^{n-r}\begin{bmatrix}n-1\\r-1\end{bmatrix}_t
+\begin{bmatrix}n-1\\r\end{bmatrix}_t\label{binom}
\end{eqnarray}
\end{lem}
\begin{proof}
\begin{eqnarray}
RHS&=&\frac{1-t^{n-1}}{1-t}\cdots\frac{1-t^{n-r+1}}{1-t^{r-1}}
 \Big(1+t^r\frac{1-t^{n-r}}{1-t^r}\Big)\nonumber\\\nonumber
&=& \frac{1-t^{n-1}}{1-t}\cdots\frac{1-t^{n-r+1}}{1-t^{r-1}}\frac{1-t^n}{1-t^r}
= LHS.\quad\nonumber
\end{eqnarray}
Another identity can be proved similarly. \qed
\end{proof}

This formula above characterizes the t-binomial. That is, if 
$\begin{bmatrix}n\\0\end{bmatrix}=\begin{bmatrix}n\\n\end{bmatrix}=1$ and the relation ($\ref{bino}$)
or ($\ref{binom}$) is true,
all t-binomials are uniquely determined inductively. Therefore, if a function $F(n,r)$
satisfies ($\ref{bino}$), ($\ref{binom}$)  and $F(n,n)=F(n,0)=1$, 
we can affirm $F(n,r)=\begin{bmatrix}n\\r\end{bmatrix}$.
From now on, again we deal with the scalar parts of Macdonald operators.
\begin{lem}
$$
\sum_{I\subset \{1,\cdots,n\},|I|=r}\prod_{i\in I,j\notin I}\frac{tx_i-x_j}{x_i-x_j}
=\begin{bmatrix}n\\r\end{bmatrix}
$$
\end{lem}
\begin{proof}
We call LHS $F(n,r)$. First, $F(n,0)=F(n,n)=1$. Next, we observe LHS as the 
function of $x_1$. The residue at $x_1=x_2$ is zero because if there is the 
set I which satisfies $1\in I, 2\notin I$, there is always the set $I'$ with the 
property $I\setminus\{1\}= I'\setminus \{2\}$. By combining these two parts, the 
residues kill each other as
\begin{eqnarray*}
\lim_{x_1\to x_2}(x_1-x_2)\Big\{\frac{tx_1-x_2}{x_1-x_2}\prod_{k\notin I,k\neq 2}
\frac{tx_1-x_k}{x_1-x_k}\prod_{i\in I\setminus\{1\},j\notin I}\frac{tx_i-x_j}{x_i-x_j}\\
-\frac{tx_2-x_1}{x_2-x_1}\prod_{k\notin I',k\neq 1}
\frac{tx_2-x_k}{x_2-x_k}
\prod_{i\in I'\setminus\{2\},j\notin I'}\frac{tx_i-x_j}{x_i-x_j}\Bigr\}\\
=(t-1)x_2\prod_{k\notin I,k\neq 2}
\frac{tx_2-x_k}{x_2-x_k}\Bigl(\prod_{i\in I\setminus\{1\},j\notin I}\frac{tx_i-x_j}{x_i-x_j}
-\prod_{i\in I'\setminus\{2\},j\notin I}\frac{tx_i-x_j}{x_i-x_j}\Bigr)=0.
\end{eqnarray*}

 Because of the symmetricity of the function, the residues 
of $x_i=x_j (i,j\in\{1,\cdots,n\})$ are all equal to zero. Taking the limit $x_1\to\infty$,
by the fact that $\frac{tx_1-x_j}{x_1-x_j}\to t,\quad \frac{tx_j-x_1}{x_j-x_1}\to 1$,
LHS becomes $t^{n-r}F(n-1,r-1)+F(n-1,r)$. $F(n,r)$ is holomorphic function of $x_1$ and 
there exists the limit of $x_1\to\infty$. So $F(n,r)$ is equal to this 'constant'
$t^{n-r}F(n-1,r-1)+F(n-1,r)$. (To affirm this is equal to a constant, The supposition 
of the induction is necessary.) Because $F(n,r)$ satisfies the relation ($\ref{binom}$),
and $F(n,n)=F(n,0)=1$,
$F(n,r)$ equals to $\begin{bmatrix}n\\r\end{bmatrix}$. \qed
\end{proof}
\begin{lem}
If we suppose $t=exp(\beta h)$, the Taylor expansion of $\begin{bmatrix}n\\r\end{bmatrix}_t$
is equal to 
\begin{eqnarray*}
\left(\begin{matrix}n\\r\end{matrix}\right)&+&\frac{r}{2}\left(\begin{matrix}n\\r\end{matrix}\right)
(n-r)h\beta\\
&+&\frac{r}{24}\left(\begin{matrix}n\\r\end{matrix}\right)(n-r)((3r+1)n-3r^2+1)\beta^2h^2\\
&+&\frac{r^2(r+1)}{48}\left(\begin{matrix}n\\r\end{matrix}\right)(n-r+1)(n-r)^2\beta^3h^3+O(h^4).
\end{eqnarray*}
\end{lem}
\begin{proof}
The constant term is easily justified by $\lim_{t\to1}\frac{1-t^m}{1-t^n}=\frac{m}{n}$.About the 
higher order of $h$, using the relation 
\begin{equation}
\begin{bmatrix}n+1\\r\end{bmatrix}=\begin{bmatrix}n\\r-1\end{bmatrix}
+t^r\begin{bmatrix}n\\r\end{bmatrix},\label{bino2}
\end{equation}
 we can prove the expansion by induction of 
$n$ and $r$. In the case $r=n$ and $r=0$,the lemma is trivially true. To verify the coefficient of $h^1$,
supposing that $a_{n,r}=\frac{r}{2}\left(\begin{matrix}n\\r\end{matrix}\right)(n-r)\beta$, we 
have to prove the  equation
$$
a_{n+1,r}=a_{n,r-1}+r\beta \left(\begin{matrix}n\\r\end{matrix}\right)+ a_{n,r}.
$$
These equations are obtained by comparing the coefficient of $h^1$'s coefficient of the relation
$(\ref{bino2})$. and we can ascertain the RHS is equal to 
$\frac{r}{2}\left(\begin{matrix}n+1\\r\end{matrix}\right)(n+1-r)\beta$.

With respect to $h^2$, supposing
$$
a_{n,r}=\frac{r}{24}\left(\begin{matrix}n\\r\end{matrix}\right)(n-r)((3r+1)n-3r^2+1)\beta^2
$$
we have to verify 
$$
a_{n+1,r}=a_{n,r-1}+\frac{\beta^2 r^2}{2} \left(\begin{matrix}n\\r\end{matrix}\right)
(n-r+1)+ a_{n,r}.
$$
Then,
\begin{eqnarray*}
RHS&=&\left(\begin{matrix}n\\r-1\end{matrix}\right)
\frac{n-r+1}{24}\beta^2\Bigl\{(r-1)\{n(3r-2)-(3r^2-6r+2)\}\\
&+&(n-r)\{n(3r+1)-(3r^2-1)\}+(n-r+1)12r\Bigr\}\\
&=&\beta^2\left(\begin{matrix}n+1\\r\end{matrix}\right)\frac{r}{24}(n+1-r)\bigl((3r+1)(n+1)-(3r^2-1)\bigr)\\
&=&a_{n+1,r}.
\end{eqnarray*}
About the coefficient of $h^3$, supposing 
$$
a_{n,r}=\frac{r^2(r+1)}{48}\left(\begin{matrix}n\\r\end{matrix}\right)
(n-r+1)(n-r)^2\beta^3
$$ 
we have to justify
$$
a_{n+1,r}=a_{n,r-1}+a_{n,r}+\frac{r^2}{24}\left(\begin{matrix}n\\r\end{matrix}\right)
\bigl((3r+1)n-3r^2+3r\bigr)(n-r+1)\beta^3.
$$  
The calculation of RHS is 
\begin{eqnarray*}
RHS&=&\frac{r^2(r-1)^2}{48}\left(\begin{matrix}n\\r\end{matrix}\right)(n-r+2)(n-r+1)\beta^3\\
&+&\frac{r^2(r+1)}{48}\left(\begin{matrix}n\\r\end{matrix}\right)(n-r+1)(n-r)^2\beta^3\\
&+&\frac{r^2}{24}\left(\begin{matrix}n\\r\end{matrix}\right)
\bigl((3r+1)n-3r^2+3r\bigr)(n-r+1)\beta^3\\
&=&\beta^3\frac{r^2}{48}\left(\begin{matrix}n\\r\end{matrix}\right)(n-r+1)(r+1)
\bigl(n^2-(r-3)n-r+2\bigr)\\
&=&\beta^3\frac{r^2(r+1)}{48}\left(\begin{matrix}n+1\\r\end{matrix}\right)(n-r+2)(n-r+1)^2
=a_{n+1,r}
\end{eqnarray*}
By the calculation above and the induction,the proof is completed. \qed 
\end{proof}

We can easily calculate the Taylor expansion of $t^{\frac{r(r-1)}{2}}\begin{bmatrix}n\\r\end{bmatrix}$
using the former lemma.
\begin{cor}
On the condition that $t=exp(\beta h)$, the function 
$t^{\frac{r(r-1)}{2}}\begin{bmatrix}n\\r\end{bmatrix}$ has following Taylor expansion.
\begin{eqnarray*}
t^{\frac{r(r-1)}{2}}\begin{bmatrix}n\\r\end{bmatrix}&=&
\left(\begin{matrix}n\\r\end{matrix}\right)+\frac{r}{2}\left(\begin{matrix}n\\r\end{matrix}\right)(n-1)
\beta h\\ &+&\frac{r}{24}\left(\begin{matrix}n\\r\end{matrix}\right)\bigl((3r+1)n^2+(1-7r)n+2r\bigr)
\beta^2h^2\\ &+&\frac{r^2}{48}\left(\begin{matrix}n\\r\end{matrix}\right)
n(n-1)\bigl((r+1)n+1-3r\bigr)\beta^3h^3+O(h^4).
\end{eqnarray*}
\end{cor}
\begin{rem}
We repeat that the function $t^{\frac{r(r-1)}{2}}\begin{bmatrix}n\\r\end{bmatrix}$ is the 'scalar part'
of the Macdonald operator $D_n^r$.
\end{rem}
\section{The Dunkl operators}
In this section we introduce k-th Dunkl operator $H_k$ and represent them explicitly.
\begin{dfn}
$k$-th Dunkl operators are defined as follows.At the moment, they are supposed to act on
the space of the functions of $x_1,\dots,x_n$.
$$
d_i=x_i\partial_i+\beta\sum_{j\neq i}\frac{x_i}{x_i-x_j}(1-K_{i,j}),\quad
H_k=\sum_{i=1}^n d_i^k
$$
$$
K_{i,j}f(x_1,\cdots,x_i,\cdots,x_j,\cdots,x_n)=f(x_1,\cdots,x_j,\cdots,x_i,\cdots,x_n)
$$
\end{dfn}
These operators commute each other when they are regulated to the space of symmetric functions.
(cf.$\cite{DW}$).
\begin{thm}
The operators $H_k$ commute each other when regulated to the space of symmetric functions.
$$
[H_i,H_j]=0,\quad (1\leq i,j\leq n).
$$
\end{thm}
Now, we calculate the explicit forms of them, more precisely, express them only by the derivations
and rational functions.
\begin{lem}
$$H_1=\sum_{i=1}^nx_i\partial_i,$$
$$H_2=\sum_{i=1}^n(x_i\partial_i)^2+
\beta\sum_{i<j}\frac{x_i+x_j}{x_i-x_j}(x_i\partial_i-x_j\partial_j)$$
\end{lem}
\begin{proof}
All we have to do is, to move all $K_{i,j}$ to the left edge and regard it as the identity 
operator since they act on the symmetric functions. The formula is
$$
K_{i,j}f(x_1,\cdots,x_i,\cdots,x_j,\cdots,x_n)=f(x_1,\cdots,x_j,\cdots,x_i,\cdots,x_n)K_{i,j},
$$
$$
K_{i,j}\partial_i=\partial_jK_{i,j}.
$$
In the first one, the function $f$ is regarded as the multiplication operator. Now, we proceed to the
calculation. The $H_1$'s form is trivial. For $H_2$, $d_i^2$ are calculated as 
\begin{eqnarray*}
\Bigl(x_i\partial_i+\beta\sum_{j\neq i}\frac{x_i}{x_i-x_j}(1-K_{i,j})\Bigr)\Bigl(x_i\partial_i\Bigr)\\
=(x_i\partial_i)^2+\beta\sum_{j\neq i}\frac{x_i}{x_i-x_j}(x_i\partial_i-x_j\partial_j)
\end{eqnarray*}
This part contains $\frac{x_i}{x_i-x_j}(x_i\partial_i-x_j\partial_j)$ and $d_j^2$ contains 
$\frac{x_j}{x_j-x_i}(x_j\partial_j-x_i\partial_i)$. When we combine them, we can obtain the term
$\frac{x_i+x_j}{x_i-x_j}(x_i\partial_i-x_j\partial_j)$. Therefore the result is immediately ascertained.
\qed
\end{proof}

For convenience, we introduce following denotations.
\begin{dfn}
We define the operators $B_{k,l},\ L_i,\ m$ as
$$
B_{k,l}=\sum_{i_1<i_2<\cdots<i_k}\sum_{s=1}^k
\frac{x_{i_s}^{k-1}(x_{i_s}\partial_{i_s})^l}{\prod_{t\neq s}(x_{i_s}-x_{i_t})}
$$
$$
L_k=\sum_{i=1}^n(x_{i}\partial_{i})^k\quad m_{1,1}=\sum_{i<j}(x_{i}\partial_{i})(x_{j}\partial_{j}),
$$
$$
m_{2,1}=\sum_{i<j}\bigl((x_{i}\partial_{i})^2(x_{j}\partial_{j})
+(x_{i}\partial_{i})(x_{j}\partial_{j})^2\bigr)
$$
$$
m_{1,1,1}=\sum_{i<j<k}\Bigl((x_{i}\partial_{i})(x_{j}\partial_{j})(x_{k}\partial_{k})\Bigr)
$$
\end{dfn}
$L_k$ and $m$  are what we call monomial symmetric functions of $x_i\partial_i$s.
 
\noindent{{\bf{(examples)}}

For 3-variables,
\begin{eqnarray*}
B_{2,1}&=&\frac{x_1^2\partial_1}{x_1-x_2}+\frac{x_2^2\partial_2}{x_2-x_1}
+\frac{x_1^2\partial_1}{x_1-x_3}+\frac{x_3^2\partial_3}{x_3-x_1}
+\frac{x_2^2\partial_2}{x_2-x_3}+\frac{x_3^2\partial_3}{x_3-x_2}\\
&=&\frac{x_1^2\partial_1-x_2^2\partial_2}{x_1-x_2}+
\frac{x_1^2\partial_1-x_3^2\partial_3}{x_1-x_3}+
\frac{x_2^2\partial_2-x_3^2\partial_3}{x_2-x_3}
\end{eqnarray*}

For 4-variables,
\begin{eqnarray*}
B_{3,2}&=&\frac{x_1^2(x_1\partial_1)^2}{(x_1-x_2)(x_1-x_3)}
+\frac{x_2^2(x_2\partial_2)^2}{(x_2-x_3)(x_2-x_1)}
+\frac{x_3^2(x_3\partial_3)^2}{(x_3-x_2)(x_3-x_1)}\\
&+&\frac{x_1^2(x_1\partial_1)^2}{(x_1-x_3)(x_1-x_4)}
+\frac{x_3^2(x_3\partial_3)^2}{(x_3-x_1)(x_3-x_4)}
+\frac{x_4^2(x_4\partial_4)^2}{(x_4-x_1)(x_4-x_3)}\\
&+&\frac{x_1^2(x_1\partial_1)^2}{(x_1-x_2)(x_1-x_4)}
+\frac{x_2^2(x_2\partial_2)^2}{(x_2-x_1)(x_2-x_4)}
+\frac{x_4^2(x_4\partial_4)^2}{(x_4-x_1)(x_4-x_2)}\\
&+&\frac{x_2^2(x_2\partial_2)^2}{(x_2-x_3)(x_2-x_4)}
+\frac{x_3^2(x_3\partial_3)^2}{(x_3-x_2)(x_3-x_4)}
+\frac{x_4^2(x_4\partial_4)^2}{(x_4-x_2)(x_4-x_3)}\\
\end{eqnarray*}
According to these denotations, we can express $H_1$ and $H_2$ as
$$
H_1=L_1,\quad H_2=L_2+2\beta B_{2,1}-\beta(n-1)H_1.
$$
This is the result of this identity
$$
\frac{x_1+x_2}{x_1-x_2}(x_1\partial_1-x_2\partial_2)
+\frac{x_1-x_2}{x_1-x_2}(x_1\partial_1+x_2\partial_2)
=\frac{2(x_1^2\partial_1-x_2^2\partial_2)}{x_1-x_2}.
$$
The following lemma is needed to summerize the $\beta^2$'s coefficient of $H_3$.
In fact, this is nothing but the operator appearing as the coefficient of $\beta^2$ in $H_3$. 
\begin{lem}
\begin{eqnarray*}
&&\sum_{i=1}^n\Bigl\{\bigl(\sum_{j\neq i}\frac{x_i}{x_i-x_j}(1-K_{i,j})\bigr)
\bigl(\sum_{j\neq i}\frac{x_i}{x_i-x_j}(x_i\partial_i-x_j\partial_j)\bigr)\Bigr\}\\
&=&(-n+3)\sum_{i<j}\frac{x_i+x_j}{x_i-x_j}(x_i\partial_i-x_j\partial_j)+6B_{3,1}-(n-1)(n-2)L_1\\
&=&2(-n+3)B_{2,1}+6B_{3,1}-(n-1)L_1.
\end{eqnarray*}
\end{lem}
\begin{proof}
We fix $i$. First,we consider the part which contains at most two variables,
more precisely the term like
$$
\frac{x_i}{x_i-x_l}(1-K_{i,l})\frac{x_i}{x_i-x_l}(x_i\partial_i-x_l\partial_l).
$$
This term is calculated as
$$
\frac{x_i}{x_i-x_l}\frac{x_i-x_l}{x_i-x_l}(x_i\partial_i-x_l\partial_l)=
\frac{x_i}{x_i-x_l}(x_i\partial_i-x_l\partial_l).
$$
When we sum them up for all $i$, there is also the part 
$\frac{x_l}{x_l-x_i}(x_l\partial_l-x_i\partial_i)$. By combining these parts, we
can see 
\begin{equation}
\sum_{i<j}\frac{x_i+x_j}{x_i-x_j}(x_i\partial_i-x_j\partial_j)\label{a}
\end{equation}
appear. Next,we have to deal with the part containing three variables. 
Again $i$ is fixed, then those parts are 
\begin{eqnarray*}
\frac{x_i}{x_i-x_j}(1-K_{i,j})\frac{x_i}{x_i-x_k}(x_i\partial_i-x_k\partial_k)\\
+\frac{x_i}{x_i-x_k}(1-K_{i,k})\frac{x_i}{x_i-x_j}(x_i\partial_i-x_j\partial_j).
\end{eqnarray*} 
This is equal to 
$$
\frac{2x_i^3\partial_i}{(x_i-x_j)(x_i-x_k)}+\frac{2x_ix_k}{(x_i-x_k)(x_j-x_k)}(x_k\partial_k)
+\frac{2x_ix_j}{(x_j-x_i)(x_j-x_k)}(x_j\partial_j)
$$
If we denote the part above as $F_{i,j,k}$, there are also $F_{j,k,i}$ and $F_{k,i,j}$ and we
try to sum up these three of $F$ and then, sum them  up for all $i,j,k$. The result is 
\begin{eqnarray}
2B_{3,1}&+&2\sum_{i<j<k}\Bigl\{\frac{1}{(x_i-x_j)(x_i-x_k)(x_j-x_k)}\nonumber\\
&\times&\bigl(x_i(x_j^2-x_k^2)x_i\partial_i+x_j(x_k^2-x_i^2)x_j\partial_j
+x_k(x_i^2-x_j^2)x_k\partial_k\bigr)\Bigr\}\label{dun1}
\end{eqnarray} 
On the other hand, we can ascertain  
\begin{eqnarray}
&\quad&(n-2)\sum_{i<j}\frac{(x_i+x_j)}{x_i-x_j}(x_i\partial_i-x_j\partial_j)\nonumber\\
=&\sum_{i<j<k}&\Bigl(\frac{(x_i+x_j)}{x_i-x_j}(x_i\partial_i-x_j\partial_j)+
\frac{(x_i+x_k)}{x_i-x_k}(x_i\partial_i-x_k\partial_k)+
\frac{(x_j+x_k)}{x_j-x_k}(x_j\partial_j-x_k\partial_k)\Bigr)\nonumber\\
=&\sum_{i<j<k}&\Bigl\{\frac{2x_i^3\partial_i}{(x_i-x_j)(x_i-x_k)}-
\frac{2x_jx_k(x_i\partial_i)}{(x_i-x_j)(x_i-x_k)}\nonumber\\
&\quad&\quad +\frac{2x_j^3\partial_j}{(x_j-x_i)(x_j-x_k)}-
\frac{2x_ix_k(x_j\partial_j)}{(x_j-x_i)(x_j-x_k)}\nonumber\\
&\quad&\quad +\frac{2x_k^3\partial_k}{(x_k-x_i)(x_k-x_j)}-
\frac{2x_ix_j(x_k\partial_k)}{(x_k-x_i)(x_k-x_j)}\Bigr\}\nonumber\\
&=&2B_{3,1}-2\sum_{i<j<k}\Big\{\frac{1}{(x_i-x_j)(x_i-x_j)(x_j-x_k)}\nonumber\\
&\quad&\times\bigl(x_jx_k(x_j-x_k)x_i\partial_i+x_ix_k(x_k-x_i)x_j\partial_j
+x_ix_j(x_i-x_j)x_k\partial_k\bigr)\Bigr\}.\label{two}
\end{eqnarray}
The numerator of the second term in ($\ref{dun1}$) can be adjusted as
\begin{eqnarray*}
numerator&=&(x_j-x_k)(x_i^2-x_i^2+x_ix_j+x_ix_k-x_jx_k+x_jx_k)x_i\partial_i\\
&+&(x_k-x_i)(x_j^2-x_j^2+x_jx_k+x_jx_i-x_ix_k+x_ix_k)x_j\partial_j\\
&+&(x_i-x_j)(x_k^2-x_k^2+x_kx_i+x_kx_j-x_ix_j+x_ix_j)x_k\partial_k\\
&=&(x_j-x_k)(x_i^2-(x_i-x_j)(x_i-x_k)+x_jx_k)x_i\partial_i\\
&+&(x_k-x_i)(x_j^2-(x_j-x_i)(x_j-x_k)+x_ix_k)x_j\partial_j\\
&+&(x_i-x_j)(x_k^2-(x_k-x_i)(x_k-x_j)+x_ix_j)x_k\partial_k
\end{eqnarray*}
This adjustment enables us to calculate the second term in the ($\ref{dun1}$) as
\begin{eqnarray}
&\quad&2B_{3,1}-2\sum_{i<j<k}\bigl\{x_i\partial_i+x_j\partial_j+x_k\partial_k\bigr\}\nonumber\\
&+&2\sum_{i<j<k}\Bigl\{\frac{1}{(x_i-x_j)(x_i-x_k)(x_j-x_k)}\nonumber\\
&\times&\bigl(x_jx_k(x_j-x_k)x_i\partial_i+x_ix_k(x_i-x_k)x_j\partial_j+
x_ix_j(x_i-x_j)x_k\partial_k\bigr)\Bigr\}\nonumber\\
&=&2B_{3,1}-(n-1)(n-2)L_1+2B_{3,1}-(n-2)\sum_{i<j}
\frac{x_i+x_j}{x_i-x_j}(x_i\partial_i-x_j\partial_j)\label{i}
\end{eqnarray}
The last equality is justified by the calculation ($\ref{two}$). The sum of
($\ref{a}$),($\ref{i}$),and the first term of ($\ref{dun1}$) is equal to
$$
6B_{3,1}+(3-n)\sum_{i<j}\frac{x_i+x_j}{x_i-x_j}(x_i\partial_i-x_j\partial_j)-(n-1)(n-2)L_1.
$$ 
By the identity
$$
\sum_{i<j}\frac{x_i+x_j}{x_i-x_j}(x_i\partial_i-x_j\partial_j)=2B_{2,1}-(n-1)L_1.
$$
we can affirm 
\begin{eqnarray*}
&&(3-n)\sum_{i<j}\frac{x_i+x_j}{x_i-x_j}(x_i\partial_i-x_j\partial_j)-(n-1)(n-2)L_1\\
&=&2(3-n)B_{2,1}-(3-n)(n-1)L_1-(n-1)(n-2)L_1= 2(3-n)B_{2,1}-(n-1)L_1.
\end{eqnarray*}
Therefore the lemma is proved.\qed
\end{proof}

Now, we can proceed to the stage to express the third Dunkl operator.
\begin{prop}
$H_3$ is expressed as 
\begin{eqnarray*}
H_3=L_3+\beta(3B_{2,2}-(n-1)L_2-m_{1,1})\\ 
+\beta^2(2(3-n)B_{2,1}+6B_{3,1}-(n-1)L_1)
\end{eqnarray*}
\end{prop}
\begin{proof}

The coefficient of $\beta^2$ is already justified by the former lemma. 
Then, we calculate the rest. The rest is 
\begin{eqnarray*}
\sum_{i}\Bigl\{(x_i\partial_i)\bigl((x_i\partial_i)^2
+\beta\sum_{j\neq i}\frac{x_i}{x_i-x_j}(x_i\partial_i-x_j\partial_j)\bigr)\Bigr\}\\
+\sum_{i}\Bigl(\bigl(\beta\sum_{j\neq i}\frac{x_i}{x_i-x_j}(1-K_{i,j})\bigr)
\bigl((x_i\partial_i)^2\bigr)\Bigr)
\end{eqnarray*}  
\begin{eqnarray*}
&=&\sum_{i}(x_i\partial_i)^3+\beta\Bigl(\sum_{i}
\sum_{j\neq i}\frac{(-x_ix_j)}{(x_i-x_j)^2}(x_i\partial_i-x_j\partial_j)\\
&&+\sum_{i}\sum_{j\neq i}\frac{x_i}{x_i-x_j}((x_i\partial_i)^2-x_i\partial_ix_j\partial_j)\\
&&+\sum_{i}\sum_{j\neq i}\frac{x_i}{x_i-x_j}((x_i\partial_i)^2-(x_j\partial_j)^2)\Bigr).
\end{eqnarray*}
We remark the following identity
$$
\sum_{i}
\sum_{j\neq i}\frac{x_ix_j}{(x_i-x_j)^2}(x_i\partial_i-x_j\partial_j)=0.
$$
And if we always try to combine the two terms containing $\frac{1}{x_i-x_j}$,the result of the 
calculation is
$$
L_3+\beta\Bigl(3B_{2,2}-(n-1)L_2-m_{1,1}\Bigl).
$$
(As the former calculation, the identity 
$$
\sum_{i<j}\frac{x_i+x_j}{x_i-x_j}((x_i\partial_i)^2-(x_j\partial_j)^2)=2B_{2,2}-(n-1)L_2
$$ is used.)
Then, the proposition is proved.\qed
\end{proof}

\section{The first and second order of $D_n^r$}
After several preliminaries above, finally we proceed to the $h$-expansion of
Macdonald operators.
\begin{dfn}

The Macdonald operators $D_n^r$ $(1 \leq r\leq n)$ are defined as
\begin{eqnarray*}
D_n^r=t^{r(r-1)/2}\sum_{|I|=r}\prod_{i\in I,j\notin I}\frac{t x_i-x_j}{x_i-x_j}
\prod_{i\in I} T_{q,x_i}\\ 
T_{q,x_i}f(x_1,\cdots,x_i,\cdots,x_n)=f(x_1,\cdots,qx_i,\cdots,x_n).
\end{eqnarray*}
The set $I$ runs all the subsets of $\{1,2,\cdots,n\}$ containing $r$ elements.
\end{dfn}

\noindent{{\bf{(examples)}}

$$
D_2^1=\frac{tx_1-x_2}{x_1-x_2}T_{q,x_1}+\frac{tx_2-x_1}{x_2-x_1}T_{q,x_2}
$$
\begin{eqnarray*}
D_3^2=&t&\times\Bigl\{\frac{tx_1-x_3}{x_1-x_3}\frac{tx_2-x_3}{x_2-x_3}T_{q,x_1}T_{q,x_2}\\
&+&\frac{tx_1-x_2}{x_1-x_2}\frac{tx_3-x_2}{x_3-x_2}T_{q,x_1}T_{q,x_3}\\
&+&\frac{tx_2-x_1}{x_2-x_1}\frac{tx_3,-x_1}{x_3-x_1}T_{q,x_2}T_{q,x_3}\Bigl\}
\end{eqnarray*}
$D_n^r$ acts on the space of the symmetric functions of which coefficient is $\bf{Q}(q,t)$.
Macdonald operators $D_n^r$ have following important property of commutativity \cite{M}.
\begin{thm}
The operators $D_n^r$ commute each other.
\begin{equation*}
\bigr[D_n^r,D_n^s\bigl]=0.
\end{equation*}
\end{thm}
Now we are engaged in their $h$-expansion supposing that $q=exp(h)$ and $t=exp(\beta h)$, around
$h=0$. That is, we have to calculate the coefficient of $h^k$ ($k=1,2,3$) of the following operator
as we previewed in the introduction of $h$-expansion.
\begin{eqnarray*}
\Bigl(1+\frac{r(r-1)}{2}\beta h+\frac{r^2(r-1)^2}{8}\beta^2h^2
+\frac{r^3(r-1)^3}{48}\beta^3h^3+O(h^4)\Bigr)\\
\times\sum_{I}\Bigl\{\prod_{i\in I,j\notin I}(1+\frac{x_i}{x_i-x_j}\sum_{k=1}^{\infty}\frac{\beta^kh^k}{k!})
\prod_{i_c\in I}\bigl(1+\sum_{k=1}^{\infty}
\frac{(x_{i_1}\partial_{x_{i_1}}+\cdots+x_{i_r}\partial_{x_{i_r}})^k}{k!}\bigr)\Bigr\}
\end{eqnarray*}
When we calculate the coefficient of $h^m$, the scalar part (the part without any derivations) is 
only comes from $t^{r(r-1)/2}\begin{bmatrix}n\\r\end{bmatrix}$ as mentioned before. In fact, this scalar 
part is the coefficient of $\beta^m$ and is already calculated in the section 2. From now on, we call 
the coefficient of $h^k$ the $k$-th order.
\begin{prop}
The first order of $D_n^r$ is 
\begin{equation*}
\left(\begin{matrix}n-1\\r-1\end{matrix}\right)L_1
+\frac{\beta r}{2}\left(\begin{matrix}n\\r\end{matrix}\right)(n-r).
\end{equation*}
\end{prop}
\begin{proof}
The coefficient of $\beta^1$ is already calculated in the section 3.
From the operator
$$
T_{q,x_1}\cdots T_{q,x_r},
$$
as the coefficient of $h$, $x_1\partial_1+\cdots+x_r\partial_r$ appears. Since 
the number of appearance of $x_1\partial_1$ is equal to 
$\left(\begin{matrix}n-1\\r-1\end{matrix}\right)$, the result follows.\qed
\end{proof}

\begin{prop}
The second order of $D_n^r$ is 
\begin{eqnarray*}
&&\frac{1}{2}\left(\begin{matrix}n-2\\r-1\end{matrix}\right)H_2+\frac{1}{2}
\left(\begin{matrix}n-2\\r-2\end{matrix}\right)H_1^2\\
&+&\beta \left(\begin{matrix}n-1\\r-1\end{matrix}\right)\frac{r(n-1)}{2}H_1\\
&+&\beta^2\frac{r}{24}\left(\begin{matrix}n\\r\end{matrix}\right)\bigl((3r+1)n^2+(1-7r)n+2r\bigr)
\end{eqnarray*}
\end{prop}
\begin{proof}
We deal with $t^{r(r-1)/2}$ and $\sum\prod\frac{tx_i-x_j}{x_i-x_j}\prod T_{q,x_i}$ apart. Once, 
we fix $I$ as the set $\{1,2,\cdots,r\}$. 
Then, the coefficient of $h^2$ of the operator part is 
\begin{eqnarray*}
\beta \Bigl\{\frac{x_1}{x_{1}-x_{r+1}}&+&\cdots+\frac{x_1}{x_{1}-x_{n}}\\
&+&\quad\frac{x_2}{x_{2}-x_{n}}\\
&+&\quad\frac{x_r}{x_{r}-x_{n}}\Bigr\}
\bigl(x_1\partial_1+\cdots+x_r\partial_r\bigr)\\
&+&\frac{1}{2}(x_1\partial_1+\cdots+x_r\partial_r)^2.
\end{eqnarray*}
From the first part the type of terms
$\frac{x_i(x_i\partial_i)}{x_i-x_j}$ are calculated as $B_{2,1}$ and the number of
these terms, when $I$ runs all subsets, equal to 
$r(n-r)\left(\begin{matrix}n\\r\end{matrix}\right)$.But to make the term
$\frac{x_i^2\partial_i-x_j^2\partial_j}{x_i-x_j}$, two terms are necessary and each
$B_{2,1}$ has $\frac{n(n-1)}{2}$ terms. Therefore the part concerning $B_{2,1}$ is finally equal to
$$
\beta r(n-r)\left(\begin{matrix}n\\r\end{matrix}\right)\frac{1}{2}\frac{2}{n(n-1)}B_{2,1}=
\beta\left(\begin{matrix}n-2\\r-1\end{matrix}\right)B_{2,1}.
$$
And the type of terms $\frac{x_i(x_k\partial_k)}{x_i-x_j}\quad (i\neq k)$ are calculated
as $L_1(=H_1)$ by combining like
$$
\frac{x_i}{x_i-x_j}(x_k\partial_k)+\frac{x_j}{x_j-x_i}(x_k\partial_k)=x_k\partial_k.
$$
When the set $I$ runs,the number of these terms is 
$r(n-r)(r-1)\left(\begin{matrix}n\\r\end{matrix}\right)$. To create $L_1$, we have to
take two terms and each $L_1$ has n terms. Therefore these type terms are summerized as
\begin{eqnarray*}
\beta r(n-r)(r-1)\left(\begin{matrix}n\\r\end{matrix}\right)\frac{1}{2}\frac{1}{n}L_1\\
=\beta\left(\begin{matrix}n-1\\r\end{matrix}\right)\frac{r(r-1)}{2}L_1.
\end{eqnarray*}
And when the subset $I$ runs,the part 
$\frac{1}{2}(x_1\partial_1+\cdots+x_r\partial_r)^2$ is finally summerized as
$$
\frac{1}{2}\left(\begin{matrix}n-1\\r-1\end{matrix}\right)L_2+
\left(\begin{matrix}n-2\\r-2\end{matrix}\right)m_{1,1}.
$$ 
Still, we have to observe the terms generated from the $h^1$ of $t^{r(r-1)/2}$ and $h^1$
of operator's part. This term is calculated like 
\begin{eqnarray*}
\beta\frac{r(r-1)}{2}\sum_{I}(x_1\partial_1+\cdots+x_r\partial_r)\\
=\beta\frac{r(r-1)}{2}\left(\begin{matrix}n-1\\r-1\end{matrix}\right)L_1.
\end{eqnarray*}
Then, the tentative result is 
\begin{eqnarray*}
&&\frac{1}{2}\left(\begin{matrix}n-1\\r-1\end{matrix}\right)L_2+
\left(\begin{matrix}n-2\\r-2\end{matrix}\right)m_{1,1}\\
&+&\beta\Bigl\{\left(\begin{matrix}n-2\\r-1\end{matrix}\right)B_{2,1}
+\frac{r(r-1)}{2}\left(\begin{matrix}n-1\\r\end{matrix}\right)L_1
+\frac{r(r-1)}{2}\left(\begin{matrix}n-1\\r-1\end{matrix}\right)L_1\Bigr\}\\
&=&\frac{1}{2}\left(\begin{matrix}n-2\\r-1\end{matrix}\right)L_2+
\frac{1}{2}\left(\begin{matrix}n-2\\r-2\end{matrix}\right)H_1^2\\
&+&\beta\Bigl\{\left(\begin{matrix}n-2\\r-1\end{matrix}\right)B_{2,1}
+\frac{r(r-1)}{2}\left(\begin{matrix}n\\r\end{matrix}\right)L_1\Bigr\}\\
&=&\frac{1}{2}\left(\begin{matrix}n-2\\r-1\end{matrix}\right)H_2+\beta
\frac{r(n-1)}{2}\left(\begin{matrix}n-1\\r-1\end{matrix}\right)H_1
+\frac{1}{2}\left(\begin{matrix}n-2\\r-2\end{matrix}\right)H_1^2. 
\end{eqnarray*}
Adding the scalar part introduced in the section 2, the calculation is completed.
(If $r=1$,the coefficient of $H_1^2$ is regarded as zero.) \qed
\end{proof}
\section{The third order of $D_n^r$}
This calculation is the main result of this paper. To clarify the procedure, we 
divide the calculation with respect to the degree of $\beta$. Again the $\beta^3$ part 
is already calculated in the section 2. After clarifying the total form of third order of $D_n^r$,
we try to express them as the polynomial of $H_k$. All of the lemmas in this section is 
about the third order of $D_n^r$.
\begin{lem}
The coefficient of $\beta^0$ is 
$$
\frac{1}{6}\left(\begin{matrix}n-3\\r-2\end{matrix}\right)\frac{n-2r}{r-1}L_3
+\frac{1}{2}\left(\begin{matrix}n-3\\r-2\end{matrix}\right)L_2H_1+\frac{1}{6}
\left(\begin{matrix}n-3\\r-3\end{matrix}\right)H_1^3.
$$
\end{lem}
\begin{proof}
This part comes only from $h^3$ of 
$\sum_{I}\prod_{i\in I}T_{q,x_i}$.Then, we fix $I$ as the set$\{1,\cdots,r\}$.
First, 
\begin{eqnarray*}
&&\frac{1}{6}\bigl(x_1\partial_1+\cdots+x_r\partial_r\bigr)^3\\
&=&\frac{1}{6}\sum (x_i\partial_i)^3+\frac{1}{2}\sum_{i\neq j}(x_i\partial_i)^2(x_j\partial_j)\\
&+&\sum_{i<j<k}(x_i\partial_i)(x_j\partial_j)(x_k\partial_k).
\end{eqnarray*}
When $I$ runs all subsets containing $r$ elements, $(x_1\partial_1)^3$ appears 
$\left(\begin{matrix}n-1\\r-1\end{matrix}\right)$ times, $(x_1\partial_1)^2(x_2\partial_2)$ 
appears $\left(\begin{matrix}n-2\\r-2\end{matrix}\right)$ times, 
$(x_1\partial_1)(x_2\partial_2)(x_3\partial_3)$ appears 
$\left(\begin{matrix}n-3\\r-3\end{matrix}\right)$ times. Therefore we can calculate the sum
as
$$
\frac{1}{6}\left(\begin{matrix}n-1\\r-1\end{matrix}\right)L_3
+\frac{1}{2}\left(\begin{matrix}n-2\\r-2\end{matrix}\right)m_{2,1}+
\left(\begin{matrix}n-3\\r-3\end{matrix}\right)m_{1,1,1}.
$$ 
By using the relations
\begin{eqnarray*}
\frac{1}{6}L_1^3&=&\frac{1}{6}L_3+\frac{1}{2}m_{2,1}+m_{1,1,1},\\
L_2H_1&=&L_3+m_{2,1},
\end{eqnarray*}
the final result is easily ascertained.\qed
\end{proof}

\begin{lem}
The coefficient of $\beta^1$ is 
\begin{eqnarray*}
&&\left(\begin{matrix}n-3\\r-2\end{matrix}\right)\frac{n-2r}{2(r-1)}B_{2,2}
+\frac{r(r-1)}{4}\left(\begin{matrix}n\\r\end{matrix}\right)L_2\\
&+&\left(\begin{matrix}n-3\\r-2\end{matrix}\right)B_{2,1}H_{1}
+\left(\begin{matrix}n-3\\r-3\end{matrix}\right)\frac{n(n-1)}{2}m_{1,1}.
\end{eqnarray*}
\end{lem}
\begin{proof}
As the preceeding proof, first we calculate without the factor $t^{r(r-1)/2}$.
Again we fix the set $I$ as $\{1,\cdots,r\}$, but all of the tentative results below
are obtained always after $I$ runs all the subsets containing
$r$ elements. We have to calculate 
\begin{eqnarray*}
&&\beta\Bigl\{\frac{x_1}{x_1-x_{r+1}}+\cdots+\frac{x_1}{x_1-x_{n}}\\
&+&\frac{x_2}{x_2-x_{r+1}}+\cdots+\frac{x_2}{x_2-x_{n}}\\
&+&\frac{x_r}{x_r-x_{r+1}}+\cdots+\frac{x_r}{x_1-x_{n}}\Bigr\}\\
&\times&\Bigl\{\frac{1}{2}\bigl((x_1\partial_1)^2+\cdots+(x_r\partial_r)^2\bigr)+
\sum_{i<j}(x_i\partial_i)(x_j\partial_j)\Bigr\}.
\end{eqnarray*}
In this proof, we suppose that $i,j,k,l$ are all distinct.The number of the terms 
$\frac{x_i}{x_i-x_k}(x_i\partial_i)^2$ is $r(n-r)\left(\begin{matrix}n\\r\end{matrix}\right)$ in all,
and every $2\left(\begin{matrix}n\\2\end{matrix}\right)$ terms corresponds to one $B_{2,2}$.
Therefore, this part is calculated as
\begin{equation}
\frac{1}{2}\left(\begin{matrix}n-2\\r-1\end{matrix}\right)B_{2,2}\label{ichi}.
\end{equation}
The number of the terms $\frac{x_i}{x_i-x_k}(x_j\partial_j)^2$ is 
$\quad r(n-r)(r-1)\left(\begin{matrix}n\\r\end{matrix}\right)$ in all. 
If they are combined like
$$
\frac{x_i}{x_i-x_k}(x_j\partial_j)^2+\frac{x_k}{x_k-x_i}(x_j\partial_j)^2=(x_j\partial_j)^2,
$$
$2n$ terms correspond to one $L_2$. So this part is equal to 
\begin{equation}
\frac{1}{4}\left(\begin{matrix}n-1\\r\end{matrix}\right)r(r-1)L_2\label{ni}.
\end{equation}
The number of the term $\frac{x_i}{x_i-x_k}(x_i\partial_i)(x_j\partial_j)$ is 
$\quad r(n-r)(r-1)\left(\begin{matrix}n\\r\end{matrix}\right)$. And by combining $2(n-2)$
terms we can obtain the term like
$$
\frac{x_i^2\partial_i-x_k^2\partial_k}{x_i-x_k}(x_1\partial_1+\cdots
+\overbrace{x_i\partial_i}+\cdots+\overbrace{x_k\partial_k}+\cdots+x_n\partial_n)
$$
(The overbrace means omitting.) And this is calculated as
\begin{eqnarray*}
&=&\frac{x_i^2\partial_i-x_k^2\partial_k}{x_i-x_k}L_1
-\frac{x_i(x_i\partial_i)^2-x_k(x_k\partial_k)^2}{x_i-x_k}\\
&&-(x_i\partial_i)(x_k\partial_k).
\end{eqnarray*}
Furthermore, every $\left(\begin{matrix}n\\2\end{matrix}\right)$ 'units' above
corresponds to one $B_{2,1}L_1-B_{2,2}-m_{1,1}$. Therefore the terms of 
$\frac{x_i}{x_i-x_k}(x_i\partial_i)(x_j\partial_j)$ are finally summerized as
\begin{equation}
\left(\begin{matrix}n-3\\r-2\end{matrix}\right)\bigl(B_{2,1}L_1-B_{2,2}-m_{1,1}\bigr)\label{san}.
\end{equation} 
About the terms $\frac{x_i}{x_i-x_k}(x_j\partial_j)(x_l\partial_l)$, the number of them
is equal to $r(n-r)\left(\begin{matrix}n\\r\end{matrix}\right)\times\frac{(r-1)(r-2)}{2}$ since for
each $i$, the number of pairing $(j,l)$ is equal to  $\frac{(r-1)(r-2)}{2}$. 
To obtain a $(x_j\partial_j)(x_l\partial_l)$, it takes two terms, and each $m_{1,1}$ has 
$\frac{n(n-1)}{2}$ terms, therefore these terms are summerized as
\begin{equation}
r(n-r)\left(\begin{matrix}n\\r\end{matrix}\right)\frac{(r-1)(r-2)}{2}
\times\frac{1}{2}\frac{2}{n(n-1)}m_{1,1}=
\frac{(r-1)(r-2)}{2}\left(\begin{matrix}n-2\\r-1\end{matrix}\right)m_{1,1}\label{yon}.
\end{equation}
Now, we proceed to the part generated by the coefficient of $h^1$ in $t^{r(r-1)/2}$ and 
the coefficient of $h^2$ in the operator part. Then, we have to deal with the terms like
$$
\beta\frac{r(r-1)}{2}\bigl\{\frac{1}{2}
\bigl(x_1\partial_1+\cdots+x_r\partial_r\bigr)^2\bigr\}.
$$
They are summerized as
\begin{equation}
\beta\frac{r(r-1)}{2}\Bigl\{\frac{1}{2}
\left(\begin{matrix}n-1\\r-1\end{matrix}\right)L_2
+\left(\begin{matrix}n-2\\r-2\end{matrix}\right)m_{1,1}\Bigr\}\label{go}.
\end{equation}
The sum of ($\ref{ichi}$),($\ref{ni}$),($\ref{san}$),($\ref{yon}$),and ($\ref{go}$)
justifies the lemma. \qed
\end{proof}
\begin{lem}
The coefficient of $\beta^2$ is 
\begin{eqnarray*}
\left(\begin{matrix}n-3\\r-2\end{matrix}\right)\frac{n-2r}{r-1}B_{3,1}
+\frac{(r-1)n+r}{2}\left(\begin{matrix}n-2\\r-1\end{matrix}\right)B_{2,1}\\
+\left(\begin{matrix}n-1\\r-1\end{matrix}\right)\frac{n(r-1)}{24}
\bigl((3r-2)n-r\bigr)L_1.
\end{eqnarray*}
\end{lem}
\begin{proof}
As the former lemmas, first we begin the calculation without the factor
$t^{r(r-1)/2}$. For simplicity, the subset $I$ are fixed as $\{1,\cdots,r\}$
but the results of the calculation are obtained always after $I$ runs all subsets
containing $r$-elements.
We have to deal with 
\begin{eqnarray*}
&\beta^2&\Bigl\{\frac{1}{2}\bigl(\frac{x_1}{x_1-x_{r+1}}
+\cdots+\frac{x_1}{x_1-x_n}\\
&+&\frac{x_r}{x_r-x_{r+1}}+\cdots+\frac{x_r}{x_r-x_n}\bigr)\\
&+&\Bigl(\sum_{i,j\in I,\ k,l\notin I}\frac{x_ix_j}{(x_i-x_k)(x_j-x_l)}\Bigr)\Bigr\}
\times \bigl(x_1\partial_1+\cdots+x_r\partial_r\bigr).
\end{eqnarray*} 
\end{proof}
The terms generated from the part like
$$
\frac{1}{2}\Bigl(\sum\frac{x_i}{x_i-x_k}\Bigr)
\times\bigl(x_1\partial_1+\cdots+x_r\partial_r\bigr)
$$
are summerized just as the preceeding lemmas. The result is 
\begin{equation}
\frac{1}{2}
\left(\begin{matrix}n-2\\r-1\end{matrix}\right)B_{2,1}
+\frac{1}{4}\left(\begin{matrix}n-1\\r\end{matrix}\right)r(r-1)L_1.\label{un}
\end{equation}
Now, we divide the terms of $\frac{x_ix_j}{(x_i-x_k)(x_j-x_l)}$ into three types and
also exhibit the number of the terms.
We suppose that $i,j,k,l$ are all distinct.

\noindent{{\bf{type-1}}}
$$\frac{x_i^2}{(x_i-x_k)(x_j-x_l)},\qquad 
r\left(\begin{matrix}n-r\\2\end{matrix}\right)\left(\begin{matrix}n\\r\end{matrix}\right)$$
\noindent{{\bf{type-2}}
$$\frac{x_ix_j}{(x_i-x_k)(x_j-x_k)},\qquad
\frac{r(r-1)}{2}(n-r)\left(\begin{matrix}n\\r\end{matrix}\right)$$
\noindent{{\bf{type-3}}
$$\frac{x_ix_j}{(x_i-x_k)(x_j-x_l)},\qquad
\frac{r(r-1)}{2}(n-r)(n-r-1)\left(\begin{matrix}n\\r\end{matrix}\right)$$
We begin with type-1 terms. If we take the product with $x_i\partial_i$, they are
calculated as one $B_{3,1}$ by combining $3\left(\begin{matrix}n\\3\end{matrix}\right)$ terms.
And if we take the product with $x_s\partial_s\quad (i\neq s)$, we can see one $L_1$ appear 
by combining $3n$ terms due to the identity
$$
\frac{x_1^2}{(x_1-x_2)(x_1-x_3)}+\frac{x_2^2}{(x_2-x_1)(x_2-x_3)}
+\frac{x_3^2}{(x_3-x_1)(x_3-x_2)}=1.
$$
We can summerize the calculation of type-1 terms as
\begin{equation}
\left(\begin{matrix}n-3\\r-1\end{matrix}\right)B_{3,1}
+\left(\begin{matrix}n-1\\r+1\end{matrix}\right)\frac{r(r^2-1)}{6}L_1.\label{deux}
\end{equation}
Then, we proceed to the type-2 terms. Just as the techniques used
in the calculation of $H_3$, we can prove 
next identity.
\begin{eqnarray*}
&\sum_{i<j<k}&\left\{\frac{x_ix_j(x_i\partial_i+x_j\partial_j)}{(x_i-x_k)(x_j-x_k)}+
\frac{x_ix_k(x_i\partial_i+x_k\partial_k)}{(x_i-x_j)(x_k-x_j)}+
\frac{x_jx_k(x_j\partial_j+x_k\partial_k)}{(x_j-x_i)(x_k-x_i)}\right\}\\
&=&-B_{3,1}+(n-2)B_{2,1}.
\end{eqnarray*}
Therefore, Since the number of  the terms like 
$\frac{x_ix_j(x_i\partial_i+x_j\partial_j)}{(x_i-x_k)(x_j-x_k)}$
is $\frac{r(r-1)}{2}(n-r)\left(\begin{matrix}n\\r\end{matrix}\right)$, 
by combining $3\left(\begin{matrix}n\\3\end{matrix}\right)$ terms together, we can
ascertain the sum is equal to
\begin{equation}
\left(\begin{matrix}n-3\\r-2\end{matrix}\right)\bigl(-B_{3,1}+(n-2)B_{2,1}\bigr).\label{trois}
\end{equation}
If we make the product like 
$$
\frac{x_ix_j}{(x_i-x_k)(x_j-x_k)}\times(x_l\partial_l),
$$
they are summerized as $L_1$ by combining $3$ parts due to the identity
$$
\frac{x_1x_2}{(x_1-x_3)(x_2-x_3)}+\frac{x_2x_3}{(x_2-x_1)(x_3-x_1)}
+\frac{x_1x_3}{(x_1-x_2)(x_3-x_2)}=1.
$$  
Since the number of the product like $\frac{x_ix_j}{(x_i-x_k)(x_j-x_k)}\times(x_l\partial_l)$
is $\frac{r(r-1)}{2}(n-r)\left(\begin{matrix}n\\r\end{matrix}\right)(r-2)$, we obtain
\begin{equation}
\frac{r(r-1)}{2}(n-r)\left(\begin{matrix}n\\r\end{matrix}\right)(r-2)/(3n)L_1
=\left(\begin{matrix}n-1\\r\end{matrix}\right)\frac{r(r-1)(r-2)}{6}L_1\label{quatre}
\end{equation}
Then, we can proceed to the type-3 terms. If we take the product with $(x_i\partial_i)$
or $(x_j\partial_j)$, we can calculate them just like 
$$
\frac{x_1x_2(x_1\partial_1)}{(x_1-x_3)(x_2-x_4)}
+\frac{x_1x_4(x_1\partial_1)}{(x_1-x_3)(x_4-x_2)}
=\frac{x_1^2\partial_1}{x_1-x_3}.
$$
Therefore these terms are summerized as $B_{2,1}$. Since the number of these terms
$\frac{x_ix_j}{(x_i-x_k)(x_j-x_l)}(x_i\partial_i)$ or $\frac{x_ix_j}{(x_i-x_k)(x_j-x_l)}(x_j\partial_j)$
is equal to $\frac{(r)(r-1)(n-r)(n-r-1)}{2}
\left(\begin{matrix}n\\r\end{matrix}\right)\times2$ and it takes $4\frac{n(n-1)}{2}$ terms to make 
one $B_{2,1}$, they are summerized as
\begin{equation}
\frac{(n-r)(n-r-1)}{2}\left(\begin{matrix}n-2\\r-2\end{matrix}\right)B_{2,1}.\label{cinq}
\end{equation}
If we take the product the type-3 terms with $(x_s\partial_s)\quad (s\neq i,j)$, 
they are summerized as $L_1$. More precisely, to make a $(x_s\partial_s)$, Since
\begin{eqnarray*}
\frac{x_1x_2}{(x_1-x_4)(x_2-x_5)}+\frac{x_4x_2}{(x_4-x_1)(x_2-x_5)}\\
+\frac{x_1x_5}{(x_1-x_4)(x_5-x_2)}+\frac{x_4x_5}{(x_4-x_1)(x_5-x_2)}=1,
\end{eqnarray*}
it takes 4 terms. On the other hand a $L_1$ has $n$ terms, therefore the sum is 
\begin{eqnarray}
\left(\begin{matrix}n\\r\end{matrix}\right)\frac{r(r-1)}{2}(n-r)(n-r-1)
\times(r-2)\times\frac{1}{4n}\times L_1\nonumber\\
=\frac{1}{8}\left(\begin{matrix}n-1\\r-3\end{matrix}\right)(n-r+2)(n-r+1)
(n-r)(n-r-1)L_1.\label{six}
\end{eqnarray}
Still, we have to calculate the terms generated by the product of $h^1$ coefficient of 
$t^{r(r-1)/2}$ and $h^2$ coefficient of the operator part. Also we have to deal with  
the terms generated by the product of $h^2$ coefficient of 
$t^{r(r-1)/2}$ and $h^1$ coefficient of the operator part. They have the form like
\begin{eqnarray*}
\frac{r(r-1)}{2}\beta \Bigl\{\beta&\sum_{I}&
\Bigl(\frac{x_1}{x_1-x_{r+1}}+\cdots+\frac{x_1}{x_1-x_n}\\
&+&\cdots \cdots\\
&+&\frac{x_r}{x_r-x_{r+1}}+\cdots+\frac{x_r}{x_r-x_n}\Bigr)
(x_1\partial_1+\cdots+x_r\partial_r)\Bigr\}\\
+\quad\frac{r^2(r-1)^2}{8}\beta^2&\times&
\Bigl\{\sum_{I}(x_1\partial_1+\cdots+x_r\partial_r)\Bigr\}.
\end{eqnarray*}
They are calculated just as preceeding lemmas. The result is 
\begin{eqnarray}
&&\frac{r(r-1)}{2}\Bigl(
\left(\begin{matrix}n-2\\r-1\end{matrix}\right)B_{2,1}
+\frac{1}{2}\left(\begin{matrix}n-1\\r-2\end{matrix}\right)
(n-r+1)(n-r)L_1\Bigr)\nonumber\\
&+&\frac{r^2(r-1)^2}{8}\Bigl(
\left(\begin{matrix}n-1\\r-1\end{matrix}\right)L_1\Bigr).\label{sept}
\end{eqnarray}
By taking the sum of ($\ref{un}$), ($\ref{deux}$), ($\ref{trois}$), ($\ref{quatre}$),
 ($\ref{cinq}$), ($\ref{six}$),and  ($\ref{sept}$), the lemma follows.\qed 

For convenience, we represent the explicit form of the third degree of $D_n^r$.
As to the coefficient of $\beta^3$, refer to the section 2 on $t$- binomials.
\begin{eqnarray*}
\frac{1}{6}\left(\begin{matrix}n-3\\r-2\end{matrix}\right)\frac{n-2r}{r-1}L_3
+\frac{1}{2}\left(\begin{matrix}n-3\\r-2\end{matrix}\right)L_2H_1+\frac{1}{6}
\left(\begin{matrix}n-3\\r-3\end{matrix}\right)H_1^3\\
+\beta\Bigl\{\left(\begin{matrix}n-3\\r-2\end{matrix}\right)\frac{n-2r}{2(r-1)}B_{2,2}
+\frac{r(r-1)}{4}\left(\begin{matrix}n\\r\end{matrix}\right)L_2\\
+\left(\begin{matrix}n-3\\r-2\end{matrix}\right)B_{2,1}H_{1}
+\left(\begin{matrix}n-3\\r-3\end{matrix}\right)\frac{n(n-1)}{2}m_{1,1}\Bigr\}\\
\end{eqnarray*}
\begin{eqnarray*}
+\beta^2\Bigl\{\left(\begin{matrix}n-3\\r-2\end{matrix}\right)\frac{n-2r}{r-1}B_{3,1}
+\frac{(r-1)n+r}{2}\left(\begin{matrix}n-2\\r-1\end{matrix}\right)B_{2,1}\\
+\left(\begin{matrix}n-1\\r-1\end{matrix}\right)\frac{n(r-1)}{24}
\bigl((3r-2)n-r\bigr)L_1\Bigr\}\\
+\beta^3\left(\begin{matrix}n\\r\end{matrix}\right)\frac{r^2n(n-1)}{48}
\bigl((r+1)n+1-3r\bigr).
\end{eqnarray*}
\section{The expression of the third order of $D_n^r$ by Dunkl operators}
In this section, we finally try to express the third order of $D_n^r$ 
by the polynomial of Dunkl operators $H_k\quad (k=1,2,3)$. As to the explicit 
form of $H_k$, refer to section 3. 
\begin{thm}
The third order of $D_n^r$ is expressed as 
\begin{eqnarray*}
&&\frac{1}{6}\left(\begin{matrix}n-3\\r-2\end{matrix}\right)\frac{n-2r}{r-1}H_3
+\frac{1}{2}\left(\begin{matrix}n-3\\r-2\end{matrix}\right)H_2H_1\\
&+&\frac{n^2(3r-1)-7rn+6r}{12(n-2)}
\left(\begin{matrix}n-2\\r-1\end{matrix}\right)\beta H_2\\
&+&\frac{1}{6}\left(\begin{matrix}n-3\\r-3\end{matrix}\right)H_1^3
+\frac{\beta}{12}\left(\begin{matrix}n-3\\r-1\end{matrix}\right)
\frac{(3r^2-3r+1)n^2+(-9r^2+6r)n+8r^2-6r}{(n-r)(n-r-1)}H_1^2\\
&+&\frac{\beta^2r}{24}\bigl((3r+1)n^2+(1-7r)n+2r\bigr)
\left(\begin{matrix}n-1\\r-1\end{matrix}\right)H_1\\
&+&\beta^3\left(\begin{matrix}n\\r\end{matrix}\right)\frac{r^2n(n-1)}{48}
\bigl((r+1)n+1-3r\bigr).
\end{eqnarray*}
\end{thm}
\begin{proof}
We begin with the subtraction of next two operators
$$
\frac{1}{6}\left(\begin{matrix}n-3\\r-2\end{matrix}\right)\frac{n-2r}{r-1}H_3,\quad
\frac{1}{2}\left(\begin{matrix}n-3\\r-2\end{matrix}\right)H_2H_1
$$ 
from the third order of $D_n^r$. Then, we can annihilate the terms 
\begin{eqnarray*}
\frac{1}{6}\left(\begin{matrix}n-3\\r-2\end{matrix}\right)\frac{n-2r}{r-1}L_3,\qquad
\frac{1}{2}\left(\begin{matrix}n-3\\r-2\end{matrix}\right)L_2H_1\\
\beta\Bigl\{\left(\begin{matrix}n-3\\r-2\end{matrix}\right)\frac{n-2r}{2(r-1)}B_{2,2}
+\left(\begin{matrix}n-3\\r-2\end{matrix}\right)B_{2,1}H_{1}\Bigr\},\\
\beta^2\left(\begin{matrix}n-3\\r-2\end{matrix}\right)\frac{n-2r}{r-1}B_{3,1}.~~~~~~~~~~~~~~~~~~~~
\end{eqnarray*}
Just after this subtraction and the calculation of 
representing the term $L_2$,$L_1^2$,and $m_{1,1}$ only by $L_2$ and $L_1$,
 we can ascertain that the coefficient of $\beta L_2$ and $\beta^2 B_{2,1}$ are equal to
$$
\frac{n^2(3r-1)-7rn+6r}{12(n-2)}
\left(\begin{matrix}n-2\\r-1\end{matrix}\right),\quad \frac{n^2(3r-1)-7rn+6r}{6(n-2)}
\left(\begin{matrix}n-2\\r-1\end{matrix}\right)
$$
respectively. Therefore by subtracting 
$$
\frac{n^2(3r-1)-7rn+6r}{12(n-2)}
\left(\begin{matrix}n-2\\r-1\end{matrix}\right)\beta H_2,
$$
we can also annihilate those two terms. Still,there is  the task of calculating 
the coefficients of 
$L_1^m(=H_1^m),\quad (m=1,2,3)$, and by clarifying these coefficients
the proof is completed.\qed
\end{proof}

\noindent{{\bf{(examples)}}

The third order of $D_n^2$ is 
\begin{eqnarray*}
\frac{n-4}{6}H_3+\frac{1}{2}H_2H_1+\frac{5n^2-14n+12}{12}\beta H_2\\
+\beta\frac{7n-10}{12}H_1^2+\beta^2\frac{(n-1)(7n^2-13n+4)}{12}L_1\\
+\beta^3\frac{n^2(n-1)^2}{24}(3n-5).~~~~~~~~~~~~~~~~~~~~~~~~~~~~
\end{eqnarray*}
The third order of $D_n^1$ is 
\begin{eqnarray*}
\frac{1}{6}H_3+\frac{2n-3}{12}\beta H_2+\frac{1}{12}\beta H_1^2~~~~~~~~~\\
+\frac{(n-1)(2n-1)}{12}\beta^2H_1+\frac{n^2(n-1)^2}{24}\beta^3.
\end{eqnarray*}
\begin{rem}
The adjustment like 
$$
\left(\begin{matrix}n-3\\r-2\end{matrix}\right)\frac{n-2r}{r-1}=
\left(\begin{matrix}n-3\\r-1\end{matrix}\right)\frac{(n-2r)}{n-r-1}
$$
is needed when we apply the result to the case $r=1$.
\end{rem}
Because the family $\{H_k\}$ is commutative, the following commutativity 
is justified. $D_n^{r}(h^i)$ means the $i$-th order of $D_n^r$.
\begin{cor}
\begin{equation*}
\Bigl[D_n^r(h^i),D_n^s(h^j)\Bigr]=0,\quad (i,j =0,1,2,3,\ 1\leq r,s\leq n).
\end{equation*}
\end{cor}
\section{The partial result on the forth order of $D_n^r$}
We can ascertain by experiments that more than forth order of $D_n^r$ doesn't 
always commute with the operators of 
the second or third order of $D_n^s$. Therefore,
 $D_n^r(h^i),\quad(i\geq 4)$ can't be expressed only by the Dunkl operators
in general. But we can ascertain the following fact on the $t$-binomial's Taylor
expansion around $h=0$.
\begin{lem}
If we suppose $t=exp(\beta h)$, the coefficient of $h^4$ of $\begin{bmatrix}n\\r\end{bmatrix}$
is 
\begin{eqnarray*}
\frac{(n-r)r\beta^4}{5760}\left(\begin{matrix}n\\r\end{matrix}\right)
\Bigl\{(15r^3+30r^2+5r-2)n^3+(-3r+1)(15r^3+25r^2-4)n^2\\
+(45r^5+30r^4-60r^3-12r^2+7r-2)n+(-15r^4+30r^2-7)\Bigr\}
\end{eqnarray*}
\end{lem}
\begin{proof}
Just as section 2, in this case we have to verify 
(now, we ignore $\beta^4$.)
$$
a_{n+1,r}=a_{n,r}+a_{n,r-1}+\frac{r^3}{48}
(n-r+1)^2(n+r+nr-r^2)\left(\begin{matrix}n\\r\end{matrix}\right)
$$
on the condition that $a_{n,r}$ is equal to the form in the lemma. If we
define $b_{n,r}$ as 
$a_{n,r}=b_{n,r}\times\frac{1}{5760}\left(\begin{matrix}n\\r\end{matrix}\right)$,
the identity that we have to verify becomes 
\begin{eqnarray*}
b_{n+1,r}(n+1)r&=&b_{n,r}(n-r)r+b_{n,r-1}r(r-1)\\
&+&120r^3(n-r+1)^2(n+r+nr-r^2).
\end{eqnarray*}
Then,
\begin{eqnarray*}
RHS&=&(n+1)r\Bigl\{(15r^3+30r^2+5r-2)n^3\\
&+&(-45r^4-15r^3+115r^2+27r-10)n^2\\
&+&(45r^5-60r^4-135r^3+128r^2+46r-16)n\\
&+&(-16r^6+45r^5+15r^4-105r^3+36r^2+24r-8)\Bigr\}
=(n+1)rb_{n+1,r}
\end{eqnarray*}
\qed
\end{proof}

By using the lemma above, and the results in the section 2, we can ascertain 
next Corollary.
\begin{cor}
On the condition that $t=exp(\beta h)$,
the coefficient of $h^4$ of $t^{r(r-1)}\begin{bmatrix}n\\r\end{bmatrix}$ is 
\begin{eqnarray*}
\frac{r\beta^4}{5760}\left(\begin{matrix}n\\r\end{matrix}\right)
\Bigl\{(15r^3+30r^2+5r-2)n^4-2(45r^3+20r^2-7r+2)n^3\\
+(125r^3-54r^2+11r-2)n^2-2r(r-1)(9r+1)n-8r^3\Bigr\}.
\end{eqnarray*}
\end{cor}

Now, the 'scalar part' or $\beta^4$-part 
of  $D_{n}^r(h^4)$ is already calculated. Then, we proceed to
the calculation of the part containing derivative operators. But 
no little amount of the calculation can be attributed to the methods
introduced in the former sections. Furthermore, because there is no 
definite goal of calculation (the result cannot be expressed only by Dunkl
operators), we only introduce the most complicated part newly needed for 
$D_{n}^r(h^4)$ in this section. The part is 
$$
\sum_{|I|=r}\Bigl\{\Bigl(\sum_{i,j,k\in I\ p,q,s \notin I}
\frac{x_ix_jx_k}{(x_i-x_p)(x_j-x_q)(x_k-x_s)}\Bigr)
\times\Bigl(\sum_{i\in I}x_i\partial_i\Bigr)\Bigl\}.
$$
Just as the calculation of $D_n^r(h^3)$, we devide the rational function part into
six types. And we also represent the number of terms after $I$ runs all subsets
containing $r$ elements.We suppose that $i,j,k,p,q,$ and $s$ are all distinct.

\noindent{{\bf{type-1}}
$$\frac{x_i^3}{(x_i-x_p)(x_i-x_q)(x_i-x_s)},\qquad
\frac{r}{6}(n-r)(n-r-1)(n-r-2)\left(\begin{matrix}n\\r\end{matrix}\right)$$
\noindent{{\bf{type-2}}}
$$
\frac{x_ix_jx_k}{(x_i-x_p)(x_j-x_p)(x_k-x_p)},\qquad
\frac{r(r-1)(r-2)}{6}(n-r)\left(\begin{matrix}n\\r\end{matrix}\right)
$$
\noindent{{\bf{type-3}}}
$$\frac{x_i^2x_j}{(x_i-x_p)(x_i-x_q)(x_j-x_s)},\qquad
\frac{r(r-1)}{2}(n-r)(n-r-1)(n-r-2)\left(\begin{matrix}n\\r\end{matrix}\right)
$$
\noindent{{\bf{type-4}}}
$$\frac{x_ix_jx_k}{(x_i-x_p)(x_j-x_p)(x_k-x_q)},\quad
\frac{(n-r)(n-r-1)}{2}r(r-1)(r-2)\left(\begin{matrix}n\\r\end{matrix}\right)
$$
\noindent{{\bf{type-5}}}
$$\frac{x_ix_jx_k}{(x_i-x_p)(x_j-x_q)(x_k-x_s)},\quad
\frac{r(r-1)(r-2)}{6}(n-r)(n-r-1)(n-r-2)\left(\begin{matrix}n\\r\end{matrix}\right)
$$
\noindent{{\bf{type-6}}}
$$\frac{x_i^2x_j}{(x_i-x_p)(x_i-x_q)(x_j-x_p)},\quad
r(r-1)(n-r)(n-r-1)\left(\begin{matrix}n\\r\end{matrix}\right)
$$
Now, we observe how the terms are summerized for each types. 
\begin{lem}
The type-1 terms are summerized as 
$$
\left(\begin{matrix}n-4\\r-1\end{matrix}\right)B_{4,1}
+\frac{(r+2)(r^3-r)}{24}\left(\begin{matrix}n-1\\r+2\end{matrix}\right)L_1.
$$
\end{lem}
\begin{proof}
If we take the product with $x_i\partial_i$, the terms are combined as
$B_{4,1}$. Each $B_{4,1}$ has $4 \left(\begin{matrix}n\\4\end{matrix}\right)$
terms, so the coefficient is 
$$
\frac{r}{6}(n-r)(n-r-1)(n-r-2)\left(\begin{matrix}n\\r\end{matrix}\right)
\times\frac{1}{4\left(\begin{matrix}n\\4\end{matrix}\right)}=
\left(\begin{matrix}n-4\\r-1\end{matrix}\right).
$$
If we take the product with $x_j\partial_j,\ (j\neq i)$, they are summerized 
as $L_1$. Because 4-pieces of $\frac{x_i^3}{(x_i-x_p)(x_i-x_q)(x_i-x_s)}$ make 
1, and $j$ can be chosen from $r-1$ numbers, the coefficient is 
$$
\frac{r}{6}(n-r)(n-r-1)(n-r-2)\left(\begin{matrix}n\\r\end{matrix}\right)
\times(r-1)\frac{1}{4n}=\frac{(r+2)(r^3-r)}{24}\left(\begin{matrix}n-1\\r+2\end{matrix}\right).
$$
\qed
\end{proof}
\begin{lem}
The type-2 terms are summerized as 
\begin{eqnarray*}
\left(\begin{matrix}n-4\\r-3\end{matrix}\right)\sum_{i<j<k<l}
\Bigl\{\frac{x_ix_jx_k(x_i\partial_i+x_j\partial_j+x_k\partial_k)}{(x_i-x_l)(x_j-x_l)(x_k-x_l)}~~~~~~~~\\
+\frac{x_ix_jx_l(x_i\partial_i+x_j\partial_j+x_l\partial_l)}{(x_i-x_k)(x_j-x_k)(x_l-x_k)}
+\frac{x_ix_kx_l(x_i\partial_i+x_k\partial_k+x_l\partial_l)}{(x_i-x_j)(x_k-x_j)(x_l-x_j)}~~~~~~~~~~~\\
+\frac{x_jx_kx_l(x_j\partial_j+x_k\partial_k+x_l\partial_l)}{(x_j-x_i)(x_k-x_i)(x_l-x_i)}\Bigr\}
+\frac{r(r-1)(r-2)(r-3)}{24}\left(\begin{matrix}n-1\\r\end{matrix}\right)L_1.
\end{eqnarray*}
\end{lem}
\begin{proof}
On the former part of the result, because each 'unit' has 
$4\left(\begin{matrix}n\\r\end{matrix}\right)$ terms, and because we have to 
take the product with one $(x_i\partial_i+x_j\partial_j+x_k\partial_k)$, the coefficient is
$$
\frac{r(r-1)(r-2)}{6}(n-r)\left(\begin{matrix}n\\r\end{matrix}\right)\times
\frac{1}{4\left(\begin{matrix}n\\4\end{matrix}\right)}
=\left(\begin{matrix}n-4\\r-3\end{matrix}\right).
$$
For $L_1$, we have to take the product with $x_l\partial_l,\ (l\neq i,j,k)$
and $l$ is taken from $r-3$ numbers. And since 4 pieces of 
$\frac{x_ix_jx_k}{(x_i-x_p)(x_j-x_p)(x_k-x_p)}$ correspond to 1, the coefficient of
$L_1$ is finally equal to 
$$
\frac{r(r-1)(r-2)}{6}(n-r)
\left(\begin{matrix}n\\r\end{matrix}\right)\times(r-3)\frac{1}{4n}
=\frac{r(r-1)(r-2)(r-3)}{24}\left(\begin{matrix}n-1\\r\end{matrix}\right).
$$\qed
\end{proof}
\begin{lem}
The terms of type-3 are calculated as 
\begin{eqnarray*}
\frac{r(r+1)(r-1)}{6}\left(\begin{matrix}n-2\\r+1\end{matrix}\right)B_{2,1}
+\frac{r(r-1)}{2}\left(\begin{matrix}n-3\\r\end{matrix}\right)B_{3,1}
+\frac{r(r^2-1)(r^2-4)}{12}\left(\begin{matrix}n-1\\r+2\end{matrix}\right)L_1.
\end{eqnarray*}
And the terms of type-4 are calculated as 
\begin{eqnarray*}
\frac{r(r-1)(r-2)}{6}\left(\begin{matrix}n-2\\r\end{matrix}\right)B_{2,1}
+\frac{(r-1)(r-2)}{2}\left(\begin{matrix}n-3\\r-1\end{matrix}\right)
\Bigl(-B_{3,1}+(n-2)B_{2,1}\Bigr)\\
+\ \frac{(r+1)r(r-1)(r-2)(r-3)}{12}\left(\begin{matrix}n-1\\r+1\end{matrix}\right)L_1.~~~~~~~~~~~~~~~~~
\end{eqnarray*}
\end{lem}
\begin{proof}
We begin with the case of type-3. If we take the product with $x_i\partial_i$,
by combining two terms, we obtain one $\frac{x_i^3\partial_i}{(x_i-x_p)(x_i-x_q)}$.
And by combining $3\times\left(\begin{matrix}n\\3\end{matrix}\right)$ 
pieces of them, we can obtain one 
$B_{3,1}$. Then, the result is 
$$
\frac{r(r-1)}{2}(n-r)(n-r-1)(n-r-2)\left(\begin{matrix}n\\r\end{matrix}\right)
\frac{1}{2}\frac{1}{3\left(\begin{matrix}n\\3\end{matrix}\right)}B_{3,1}
=\frac{r(r-1)}{2}\left(\begin{matrix}n-3\\r\end{matrix}\right)B_{3,1}.
$$
If we take the product with $x_j\partial_j$, by adding three terms together, 
we can see one $\frac{x_j^2\partial_j}{x_j-x_s}$ appear. And by combining them
$2\frac{n(n-1)}{2}$ times we can create one $B_{2,1}$. Then, its coefficient is 
$$
\frac{r(r-1)}{2}(n-r)(n-r-1)(n-r-2)\frac{1}{3}\frac{1}{2}
\frac{1}{\left(\begin{matrix}n\\2\end{matrix}\right)}
=\frac{r(r+1)(r-1)}{6}\left(\begin{matrix}n-2\\r+1\end{matrix}\right).
$$
There is still the product with $x_k\partial_k,\ (k\neq i,j)$. $k$ can be
chosen from $r-2$ numbers, and 6-pieces of $\frac{x_i^2x_j}{(x_i-x_p)(x_j-x_q)(x_j-x_s)}$
correspond to 1, and a $L_1$ has $n$ terms. Then the result is 
$$
\frac{r(r-1)}{2}(n-r)(n-r-1)(n-r-2)(r-2)\frac{1}{6n}L_1
=\frac{r(r^2-1)(r^2-4)}{12}\left(\begin{matrix}n-1\\r+2\end{matrix}\right)L_1.
$$
The calculation of the type-4 terms is almost the same as type-3. But instead of 
the term like $\frac{x_i^3\partial_i}{(x_i-x_p)(x_i-x_q)}$, we must deal with the term 
like $\frac{x_ix_j(x_i\partial_i+x_j\partial_j)}{(x_i-x_p)(x_j-x_p)}$. These terms are 
summerized as $-B_{3,1}+(n-2)B_{2,1}$. (Refer to the section 5, the calculation of type-3 terms.)
\qed
\end{proof}
\begin{lem}
The terms of type-5 are calculated as
$$
\frac{(r+1)r(r-1)(r-2)}{8}\left(\begin{matrix}n-2\\r+1\end{matrix}\right)B_{2,1}
+\frac{r(r-3)(r^2-1)(r^2-4)}{48}\left(\begin{matrix}n-1\\r+2\end{matrix}\right)L_1.
$$
\end{lem}
\begin{proof}
If we take the product with $x_i\partial_i$,$x_j\partial_j$,or $x_k\partial_k$,and
combine four terms, we can obtain the term like $\frac{x_i^2\partial_i}{x_i-x_p}$.
Next, by combining them $2\left(\begin{matrix}n\\2\end{matrix}\right)$ times, we
can see one $B_{2,1}$ appear. Therefore, this part is summerized as 
\begin{eqnarray*}
\frac{r(r-1)(r-2)}{6}(n-r)(n-r-1)(n-r-2)\left(\begin{matrix}n\\r\end{matrix}\right)
\frac{3}{4\times2\left(\begin{matrix}n\\2\end{matrix}\right)}B_{2,1}\\
=\frac{(r+1)r(r-1)(r-2)}{8}\left(\begin{matrix}n-2\\r+1\end{matrix}\right)B_{2,1}.~~~~~~~~~~~~~~~~~
\end{eqnarray*}
If we take the product with $x_m\partial_m,\ (m\neq i,j,k)$, $m$ can be chosen 
from $r-3$ numbers and 8-pieces of $\frac{x_ix_jx_k}{(x_i-x_p)(x_j-x_q)(x_k-x_s)}$ correspond
to 1. And since each $L_1$ has $n$ terms, the result is 
\begin{eqnarray*}
\frac{r(r-1)(r-2)}{6}(n-r)(n-r-1)(n-r-2)
\left(\begin{matrix}n\\r\end{matrix}\right)(r-3)\frac{1}{8n}L_1\\
=\frac{r(r-3)(r^2-1)(r^2-4)}{24}\left(\begin{matrix}n-1\\r+2\end{matrix}\right)L_1.~~~~~~~~~~~~~~~
\end{eqnarray*} 
\qed
\end{proof}
\begin{lem}
The terms of type-6 are summerized as 
\begin{eqnarray*}
\frac{5r(r+1)(r-1)(r-2)}{24}
\left(\begin{matrix}n-1\\r+1\end{matrix}\right)L_{1}~~~~~~~\\
+\left(\begin{matrix}n-4\\r-2\end{matrix}\right)\sum_{i<j<k<l}\Bigl\{\sum_{|J|=2}
\Bigl(\frac{4(\prod_{i\in J}x_i^2)(\sum_{i\in J}x_i\partial_i)}{\prod_{i\in J,p\notin J}(x_i-x_p)}~~~~~~~~~~\\
-\frac{(\prod_{i\in J}x_i)(\sum_{i\in J}x_i)
(\sum_{p\notin J}x_p)(\sum_{i\in J}x_i\partial_i)}
{\prod_{i\in J,p\notin J}(x_i-x_p)}\Bigr)\Bigr\}
\end{eqnarray*}
The set $J$ is the subset of $\{i,j,k,l\}$.
\end{lem}
\begin{proof}
First, we take the product with $x_m\partial_m,\ (m\neq i,j)$. 24 pieces of 
$\frac{x_i^2x_j}{(x_i-x_p)(x_i-x_q)(x_j-x_p)}$ amount to 5.(The set $\{i,j,p,q\}$
is fixed. After the all action of forth symmetric group, by summing  them all up,
this calculation is easily justified.) Therefore, the calculation is 
\begin{eqnarray*}
r(r-1)(n-r)(n-r-1)\left(\begin{matrix}n\\r\end{matrix}\right)\times(r-2)\frac{5}{24n}L_1\\
=\frac{5(r+1)r(r-1)(r-2)}{24}\left(\begin{matrix}n-1\\r+1\end{matrix}\right)L_1.
\end{eqnarray*}
Now, we proceed to the product with $(x_i\partial_i+x_j\partial_j)$. If we sum them up
after the action of interchanges $(i,j),(p,q),$ and $(i,j)(p,q)$, we can obtain the term 
$$
\frac{4x_i^2x_j^2(x_i\partial_i+x_j\partial_j)}{(x_i-x_p)(x_i-x_q)(x_j-x_p)(x_j-x_q)}
-\frac{x_ix_j(x_i+x_j)(x_p+x_q)(x_i\partial_i+x_j\partial_j)}{(x_i-x_p)(x_i-x_q)(x_j-x_p)(x_j-x_q)}.
$$
Then, Once the set $\{i,j,p,q\}$ is fixed, 
the pair of suffix $i$ and $j$ which appear in the numerator has 6 patterns. 
By translating this possibility into the sum  with respect to the set $J$ ($J\subset \{i,j,p,q\}$
,$|J|=2$), the representation on the lemma is justified. Since each 'unit' has 24 terms, the coefficient
is calculated as 
$$
r(r-1)(n-r)(n-r-1)\left(\begin{matrix}n\\r\end{matrix}\right)
\frac{1}{24\left(\begin{matrix}n\\4\end{matrix}\right)}=\left(\begin{matrix}n-4\\r-2\end{matrix}\right).
$$
\qed 
\end{proof}
\section{Discussions}
As we can see in the last section, $D_n^r(h^4)$ has some operators which 
cannot be expressed by $B_{i,j}$. But in the case of $r=1$, it is natural to expect that
we can  express the higher order operator only by $B_{i,j}$, $L_i$, and $m$s.  In fact,
$D_{n}^1(h^4)$ except the scalar part is expressed as
\begin{eqnarray*}
\frac{1}{24}L_4+\frac{\beta}{6}B_{2,3}
+\beta^2\Bigl(\frac{1}{4}B_{2,2}+\frac{1}{2}B_{3,2}\Bigr)\\
+\beta^3\Bigl(\frac{1}{6}B_{2,1}+B_{3,1}+B_{4,1}\Bigr).
\end{eqnarray*} 
If we are going to use the denotation of $B$ in general $r$ case, probably we have to
extend the meaning of suffix to the 'partitions'. 
To clarify the relation between these operators $B_{i,j}$ and Dunkl operators
and to express $D_{n}^r(h^k)~(k\geq 4)$ 'naturally' in some sense 
will be our next task. 
\section{Acknowledgement}
First, I owe almost all of my mathematical background 
to professor J.Shiraishi. His serious attitude for the research 
always continued to inspire me to infinite effort. Professor T.Oshima
gave me lots of worthful advices and mental support in daily seminar.
And if it were not for the private discussions with graduate students in this
faculty of mathematical sciences, this thesis could never have been completed
in this form. Again I represent sincere gratitude to all the people associated with me
on this research.

\end{document}